\documentclass[12pt,reqno]{amsart}
\usepackage[top=1in, bottom=1in, left=1in, right=1in]{geometry}
\usepackage{amssymb,mathrsfs,amsmath}
\usepackage[alphabetic]{amsrefs}
\usepackage{hyperref}

\newtheorem{theorem}{Theorem}%[section]
\newtheorem{lemma}[theorem]{Lemma}

\newtheorem{cor}[theorem]{Corollary}
\newtheorem{conj}[theorem]{Conjecture}
\theoremstyle{remark}

\newcommand {\N} {{\mathbb N}}

\newcommand {\Z} {{\mathbb Z}}
\newcommand {\D} {{\mathbb D}}

\newcommand{\ds}{\displaystyle}

\newcommand{\bv}\boldsymbol{}

\newcommand{\ignore}[1]{}

\renewcommand{\Re}{\textup{Re }}

\renewcommand{\mod}[1]{{\ifmmode\text{\rm\,(mod\,$#1$)}\else\discretionary{}{}{\hbox{ }}\rm(mod~$#1$)\fi}}

\newsymbol\dnd 232D

\begin{document}

\title{P\'{o}lya--Vinogradov and the least quadratic nonresidue}

\author{Jonathan W. Bober}
\address{Heilbronn Institute for Mathematical Research \\ Department of Mathematics, University of Bristol, Bristol, United Kingdom}
\email{{\tt j.bober@bristol.ac.uk}}

\author{Leo Goldmakher}
\address{Department of Mathematics and Statistics, Williams College, Williamstown, MA, USA}
\email{{\tt leo.goldmakher@williams.edu}}

%\author{}
%\address{aa\\aa}
%\email{{\tt a}}

%\subjclass[2010]{}
%\keywords{}

%\date{\today}

\begin{abstract}
It is well-known that cancellation in short character sums (e.g. Burgess' estimates) yields bounds on the least quadratic nonresidue. Scant progress has been made on short character sums since Burgess' work, so it is desirable to find another approach to nonresidues. In this note we formulate a new line of attack on the least nonresidue via long character sums, an active area of research. Among other results, we demonstrate that improving the constant in the P\'olya--Vinogradov inequality would lead to significant progress on nonresidues. Moreover, conditionally on a conjecture on long character sums, we show that the least nonresidue for any odd primitive character (mod $k$) is bounded by $(\log k)^{1.4}$.
\end{abstract}

\maketitle
\thispagestyle{empty}
%This is necessary, since \maketitle automatically numbers the first page!

%\setcounter{tocdepth}{1}
%\tableofcontents

%%%%%%%%%%%%%%%%%%
%%%%%%%%%%%%%%%%%%

\section{The case of the odd character}

It is a long-standing open problem (first studied by Gauss as a key step in his first proof of Quadratic Reciprocity) to determine strong bounds on the least quadratic nonresidue modulo $p$, which we denote $n_p$. The celebrated P\'{o}lya--Vinogradov bound on character sums immediately implies the bound $n_p \ll \sqrt{p} \log p$, and an innovation due to Vinogradov reduced this to ${n_p \ll p^{1/(2\sqrt{e})} \log^2 p}$. Burgess' character sum bound further improved this to ${n_p \ll_\epsilon p^{1/4 + \epsilon}}$, and applying Vinogradov's trick gives
\begin{equation}
\label{eq:StateOfTheArt}
n_p \ll_\epsilon p^{1/(4\sqrt{e}) + \epsilon} .
\end{equation}
This bound remains the state of the art.\footnote{It should be pointed out that Burgess' character sum bound has been slightly improved by Hildebrand in \cite{Hi86}. However, his improvements do not yield an improvement of (\ref{eq:StateOfTheArt}).}

From this history one might infer that the way forward is to improve Burgess' theorem by exhibiting cancellation in shorter character sums. This seems to be a difficult problem, however, as no one has made progress on short character sums in a long time. By contrast, there have been a number of recent breakthroughs on long character sums (discussed below), and it is likely that more are on the horizon. The goal of the present note is to show that cancellation in long character sums can lead to strong bounds on $n_p$. To state our results more precisely, we need some notation.  Given a character $\chi\mod{q}$, set
\begin{equation}
\label{eq:Notation}
S_\chi(t) := \sum_{n \leq t} \chi(n) , \quad
M(\chi) := \max_{t \leq q} |S_\chi(t)|, \quad
\text{and}\quad
n_\chi := \min\{ n \in \N : \chi(n) \neq 0,1\} .
\end{equation}
In this notation, the P\'{o}lya--Vinogradov inequality asserts the bound $M(\chi) \ll \sqrt{q} \log q$. Our main result is the following.
\begin{theorem}
\label{thm:MainThm}
Fix notation as in \eqref{eq:Notation}, and suppose $M(\chi) \leq \sqrt{q} \cdot \frac{\log q}{f(q)}$ for all primitive even characters $\chi\mod{q}$
of order $g$,
where $f$ is some function satisfying $f'(x) \geq 0$ for all sufficiently large $x$. Then for all odd primitive Dirichlet characters $\xi\mod{k}$
of order $g$
with $3\dnd k$ we have
\[
n_\xi \ll
\exp\bigg(
	\frac{\pi \sqrt{3}}{2(\sqrt{e} - 1)} \cdot \frac{\log k}{f(k)}
\bigg) .
\]
\end{theorem}

\noindent
In other words, improvements to the P\'olya--Vinogradov inequality (even of the implicit constant) would yield improvements of bounds on the least nonresidue. To get a feel for the strength of Theorem \ref{thm:MainThm}, we explore several consequences.

First, we consider the constant in the P\'olya--Vinogradov inequality. Hildebrand proved \cite[Corollary 5]{H88} that for all primitive even quadratic $\chi\mod{q}$ one has
\[
M(\chi) \leq \big( C + o(1) \big) 	\sqrt{q} \log q
\]
with $\ds C = \frac{\sqrt{e} - 1}{2\pi \sqrt{3e}} \approx 0.036$. Applying Theorem \ref{thm:MainThm} to Hildebrand's bound recovers the Burgess bound $n_p
\ll p^{\frac{1}{4\sqrt{e}} + o(1)}$ for all $p \equiv 3 \mod{4}$. Thus, any improvement over Hildebrand's constant --- for example, to $C = \frac{1}{9\pi} \approx 0.035$ --- would immediately yield an improvement of the Burgess exponent $\frac{1}{4\sqrt{e}}$ for primes $p \equiv 3\mod{4}$. (Hildebrand asserts, without proof, that the same is true for primes $p \equiv 1\mod{4}$; this is presumably a typo.)

If we are more optimistic, we might hope to improve P\'olya--Vinogradov further. In recent years, several infinite families of characters (namely, characters of odd order \cite{GS07,G} and characters with smooth conductor \cite{GSmooth}) have been shown to satisfy
\[
M(\chi) = o(\sqrt{q} \log q) .
\]
If we could prove such a bound for the family of all even characters
$\chi\mod{q}$, Theorem \ref{thm:MainThm} would imply
\[
n_\xi \ll k^{o(1)}
\]
for all odd $\xi\mod{k}$.

We conclude our discussion by exploring the limitations of Theorem \ref{thm:MainThm}. Working under the Generalized Riemann Hypothesis, Granville and Soundararajan have shown that
\[
    M(\chi) \le \left(\frac{2 e^\gamma}{\pi \sqrt{3}} + o(1)\right)
        \sqrt{q}\log\log q
\]
for every even primitive character\footnote{Although not explicitly stated, this is easily derived from the proof of Theorem 6 in \cite{GS07}.}, and conjecture that this can be improved to
\begin{equation}
\label{eq:ConjCharBound}
    M(\chi) \le \left(\frac{e^\gamma}{\pi \sqrt{3}} + o(1)\right)
        \sqrt{q}\log\log q .
\end{equation}
Applying Theorem \ref{thm:MainThm} yields the following.
\begin{cor}
\label{cor:ConjCor}
Suppose that the bound \eqref{eq:ConjCharBound} holds for all primitive even characters $\chi\mod{q}$. Then for all odd primitive characters
$\xi\mod{k}$ we have
\[
n_\xi \ll (\log k)^{c + o(1)} ,
\]
where $c = \frac{e^\gamma}{2(\sqrt{e} - 1)} \approx 1.37$.
\end{cor}

\noindent
Note that this is stronger than Ankeny's \cite{An} long-standing GRH bound $n_p \ll (\log p)^2$. On the other hand, we fall short of achieving the conjectured bound $n_p \ll_\epsilon (\log p)^{1 + \epsilon}$. In fact, Corollary \ref{cor:ConjCor} is the strongest consequence one could hope to derive from Theorem \ref{thm:MainThm}, since Granville and Soundararajan have unconditionally proved that equality is attained in \eqref{eq:ConjCharBound} for infinitely many primitive even $\chi\mod{q}$; this follows immediately from equation (1.8) of \cite{GS07}.

The maximum size of the least nonresidue seems to be intimately linked with the maximum size of Dirichlet $L$-functions at $1$. Given an odd character $\xi \mod{p}$ with large least nonresidue, we construct an even character $\chi \mod{3p}$ for which $S_\chi(\alpha p)$ is large for some $\alpha$. In fact, it follows from work of the first author \cite{BoberAvg} that we can take $\alpha$ very close to $1$. Since $S_\chi(p) = \frac{\sqrt{p}}{\pi \sqrt{3}} L(1,\xi)$, one expects that $L(1,\xi)$ is also large; however, we are not able to prove this. If we were, we would be able to recover the Burgess exponent $\frac{1}{4\sqrt{e}} + o(1)$ from Stephens' work \cite{stephens} on the size of $L(1, \chi)$.

At the heart of the proof of Theorem \ref{thm:MainThm} is the following result, which provides an upper bound on the least nonresidue for an odd character in terms of the maximum size of a related character sum.
\begin{theorem}
\label{thm:LowerBound}
Suppose $\xi\mod{k}$ and $\psi\mod{\ell}$ are odd primitive Dirichlet characters such that $(k,\ell) = 1$. Then
\[
\log n_\xi \leq
\frac{\pi}{2(\sqrt{e} - 1)} \cdot \frac{M(\xi\psi)}{\sqrt{k}} + O(\sqrt{\ell}) .
\]
\end{theorem}

\noindent
This demonstrates that we can improve bounds on $n_\xi$ by finding a \emph{single} twist of $\xi$ for which P\'olya--Vinogradov can be beaten. Although this seems stronger than Theorem \ref{thm:MainThm}, we have not been able to exploit the extra flexibility.

We expect that neither Theorem \ref{thm:MainThm} nor Theorem \ref{thm:LowerBound} is optimal. In Section \ref{sect:Heuristic} we give a heuristic argument which suggests the following.
\begin{conj}
\label{conj:StrongConj}
Let $\chi(n) := \xi(n) \left(\frac{n}{3}\right)$, where $\xi \mod{k}$ is an odd primitive character such that $3 \dnd k$. Then
\[
\log n_\xi \leq
\left(\frac{\pi}{e^\gamma} + o(1) \right)
	\frac{M(\chi)}{\sqrt{k}} .
\]
\end{conj}
\noindent
Combined with the conjectural bound \eqref{eq:ConjCharBound} on long character sums, this would yield
\[
\phantom{\text{for all } p \equiv 3\mod{4}} \qquad
n_p \leq (\log p)^{1+o(1)}
\qquad \text{for all } p \equiv 3\mod{4} .
\]

\noindent
\emph{Acknowledgements.}
We are grateful to John Friedlander and Soundararajan for encouragement and some helpful suggestions. We would also like to thank Enrique Trevi\~no for correcting a misunderstanding about Hildebrand's result, and the anonymous referee for meticulously reading the paper and making useful suggestions. The second author was partially supported by an NSERC Discovery Grant.

\section{Proofs of Theorems \ref{thm:MainThm} and \ref{thm:LowerBound}}

We begin by showing that Theorem \ref{thm:MainThm} is an easy consequence of Theorem \ref{thm:LowerBound}.

\begin{proof}[Proof of Theorem \ref{thm:MainThm}]
Let $h(x) := \frac{\log x}{f(x)}$, so that our hypothetical improvement of P\'{o}lya--Vinogradov becomes: for all primitive even $\chi\mod{q}$,
\[
M(\chi) \leq \sqrt{q} \cdot h(q) .
\]
Note that we may (and will) assume that $h(x) \leq \log x$, or equivalently that $f(x) \geq 1$, for all sufficiently large $x$.

Given any odd primitive characters $\xi \mod{k}$ and $\psi\mod{\ell}$ with $(k,\ell)=1$, set $\chi := \xi \psi$ and $q := k \ell$. It follows that $\chi\mod{q}$ is a primitive even character, so by Theorem \ref{thm:LowerBound} combined with our hypothesis on $M(\chi)$, we deduce
\[
\log n_\xi \leq
\frac{\pi \sqrt{\ell}}{2(\sqrt{e} - 1)} \cdot h(q) + O(\sqrt{\ell}) .
\]
We now show that $h(q) \leq h(k) + O(\ell)$. First, observe that
\[
h'(x) = \frac{1}{x \cdot f(x)} - \frac{f'(x) \log x}{f(x)^2}
\leq \frac{1}{x} ,
\]
since $f(x) \geq 1$ and $f'(x) \geq 0$. The Mean Value Theorem implies that for some $x_k \in [k,q]$,
\[
h(q) - h(k) = k(\ell - 1) \cdot h'(x_k) \leq \ell -1.
\]
It follows that
\[
\log n_\xi \leq
\frac{\pi \sqrt{\ell}}{2(\sqrt{e} - 1)} \cdot h(k) + O(\ell^{3/2}) .
\]
If $3 \dnd k$, we may make the choice $\psi(n) := \big(\frac{n}{3}\big)$, which yields
\[
n_\xi \ll
\exp\Big(\frac{\pi \sqrt{3}}{2(\sqrt{e}-1)} \cdot h(k)\Big)  =
\exp\Big(\frac{\pi \sqrt{3}}{2(\sqrt{e}-1)} \cdot \frac{\log k}{f(k)}\Big) .
\qedhere
\]
\end{proof}

We now shift our attention to Theorem \ref{thm:LowerBound}. Here is a rough sketch of our argument. In \cite{GLeven} it was shown that for any primitive odd characters $\xi\mod{k}$ and $\psi\mod{\ell}$, we have
\begin{equation}
\label{eq:Thm3Outline}
M(\xi\psi) \gtrsim \frac{\sqrt{k}}{\pi} \left|\sum_{n \leq N} \frac{\xi(n)}{n}\right|
\end{equation}
for any small $N$. To make the inequality as tight as possible, one wishes to choose $N$ which maximizes the sum on the right hand side. The natural choice $N = n_\xi$ would give
${M(\xi\psi) \gtrsim \frac{\sqrt{k}}{\pi} \log n_\xi}$,
which is of the same shape as our claim but with a weaker constant. To do better, we employ a version of Vinogradov's trick. For $n > n_\xi$ it is hard to predict the behavior of $\xi(n)$ in general; however, we know that $\xi(n) = 1$ for all $n_\xi$-smooth $n$, and there are many smooth numbers just a bit larger than $n_\xi$. Thus, taking $N$ to be slightly bigger than $n_\xi$ increases the size of the sum in \eqref{eq:Thm3Outline}. Optimizing the balance between the mysterious behavior of $\xi(n)$ and the well-understood behavior of smooth numbers, we settle on a choice of $N$ which gives the bound in Theorem \ref{thm:LowerBound}.

Before making the above outline rigorous, we state a few preliminary results. The following two estimates, developed in \cite{GLodd} and \cite{GLeven} to show that character sums can get large, significantly simplify our computations.

\begin{lemma}[see Lemma 2.2 of \cite{GLodd}]
\label{lem:Lemma1}
Let $\{a_n\}_{n\in \mathbb{Z}}$ be a sequence of complex numbers with $|a_n|\leq 1$ for all $n$, and let $x\geq 2$ be a real number.
Then
\[
\max_{\theta\in [0,1]}
\left|
	\sum_{1\leq |n|\leq x}\frac{a_n}{n} e(n\theta)
\right|
=
\max_{\theta\in [0,1]} \max_{N\leq x}
\left|
	\sum_{1\leq |n|\leq N}\frac{a_n}{n} e(n\theta)
\right|
+O(1).
\]
\end{lemma}

\begin{lemma}[see Lemma 2.1 of \cite{GLeven}]
\label{lem:Lemma2}
If $\psi \mod{m}$ is a primitive Dirichlet character, then
\[
\max_{\theta\in[0,1]} \left| \sum_{n \in \Z} b_n \psi(n) e(n\theta) \right|
	\geq
		\frac{\sqrt{m}}{\phi(m)}
			\bigg| \sum_{(n,m) = 1} b_n \bigg|
\]
for any set of complex numbers $\{b_n\}$ satisfying $\sum |b_n| < \infty$.
\end{lemma}

\noindent
Next, in order to apply our version of Vinogradov's trick, we shall require the following estimate. Given an integer $n \geq 2$,  denote the largest prime factor of $n$ by $P^+(n)$.
\begin{lemma}
\label{lem:ForVino}
For any fixed $\alpha \in [1,2]$ we have
\[
\sum_{\substack{n \leq y^\alpha \\ P^+(n) > y}} \frac{1}{n}
=
(\alpha\log \alpha - \alpha + 1) \log y + O(1) .
\]
\end{lemma}
\begin{proof}
We have
\[
\begin{split}
\sum_{\substack{n \leq y^\alpha \\ P^+(n) > y}} \frac{1}{n}
&=
\sum_{y < n \leq y^\alpha} \frac{1}{n} -
	\sum_{\substack{y < n \leq y^\alpha \\ P^+(n) \leq y}} \frac{1}{n} \\
&=
(\alpha - 1) \log y + O\Big(\frac{1}{y}\Big) -
	\sum_{\substack{y < n \leq y^\alpha \\ P^+(n) \leq y}} \frac{1}{n} .
\end{split}
\]
Thus it suffices to estimate the sum
\[
\sum_{\substack{y < n \leq y^\alpha \\ P^+(n) \leq y}} \frac{1}{n} .
\]
Applying partial summation and a change of variables, we obtain
\begin{equation}
\label{eq:MidStepInSmooth}
\sum_{\substack{y < n \leq y^\alpha \\ P^+(n) \leq y}} \frac{1}{n}
=
\frac{\Psi(y^\alpha, y)}{y^\alpha} - 1 +
(\log y) \int_1^\alpha \frac{\Psi(y^u, y)}{y^u} \, du ,
\end{equation}
where $\Psi(x,y)$ denotes the number of $y$-smooth numbers up to $x$. When $x$ is not much larger than $y$, the quantity $\Psi(x,y)$ is well-understood. For example, for fixed $\alpha \in [1,2]$ we have
\[
\frac{\Psi(y^\alpha, y)}{y^\alpha} =
1 - \log \alpha + O\Big(\frac{1}{\log y}\Big)
\]
uniformly in $y \geq 2$. (See the excellent survey article \cite{G08} for much more general results.) Employing this in \eqref{eq:MidStepInSmooth} and simplifying yields
\[
\sum_{\substack{y < n \leq y^\alpha \\ P^+(n) \leq y}} \frac{1}{n}
=
(-\alpha \log \alpha + 2\alpha - 2) \log y + O(1) .
\]
Combining this with our work above concludes the proof.
\end{proof}

\noindent
Having assembled all the ingredients, we can now prove Theorem \ref{thm:LowerBound}.

\begin{proof}[Proof of Theorem \ref{thm:LowerBound}]
Let $\chi := \xi \psi$ and let $q := k\ell$, so that $\chi\mod{q}$ is an even primitive character. In this case, P\'olya's Fourier expansion reads
\[
S_\chi(q\alpha) =
-\frac{\tau(\chi)}{2\pi i}
	\sum_{1 \leq |n| \leq q}
		\frac{\overline{\chi}(n)}{n} e(-n\alpha) + O(\log q) ,
\]
where $\tau(\chi)$ is the Gauss sum.
It follows that
\begin{equation}
\label{eq:MaxlEst}
M(\chi)  =
\frac{\sqrt{q}}{2 \pi} \max_\alpha \left|
\sum_{1 \leq |n| \leq q} \frac{\overline{\chi}(n)}{n} e(-n\alpha)
\right| + O(\sqrt{q}) .
\end{equation}
Pick any $N \leq q$.
Lemmas \ref{lem:Lemma1} and \ref{lem:Lemma2} imply
\[
\begin{split}
\max_\alpha \left|
\sum_{1 \leq |n| \leq q} \frac{\overline{\chi}(n)}{n} e(-n\alpha)
\right|
& \geq
\max_\alpha \left|
\sum_{1 \leq |n| \leq N} \frac{\overline{\chi}(n)}{n} e(-n\alpha)
\right| + O(1) \\
& \geq
\frac{\sqrt{\ell}}{\varphi(\ell)}
\left|
\sum_{\substack{1 \leq |n| \leq N \\ (n,\ell) = 1}}
	\frac{\overline{\xi}(n)}{n}
\right| + O(1) .
\end{split}
\]
Plugging this into \eqref{eq:MaxlEst} gives
\begin{equation}
\label{eq:LowerBdOnM1}
M(\chi)
\geq
\frac{\sqrt{k}}{\pi} \cdot \frac{\ell}{\varphi(\ell)}
\left|
\sum_{\substack{n \leq N \\ (n,\ell) = 1}}
	\frac{\xi(n)}{n}
\right|
+ O(\sqrt{q}) .
\end{equation}
It remains to choose $N$ to maximize the sum on the right hand side.

Write $N = y^\alpha$, where $y \leq n_\xi$ and $\alpha \in [1,2]$ will be determined later (note that $y^\alpha \leq q$). We are trying to find a lower bound on the magnitude of
\begin{equation}
\label{eq:LowerBdOnM2}
\sum_{\substack{n \leq y^\alpha \\ (n,\ell) = 1}} \frac{\xi(n)}{n}
=
\sum_{\substack{n \leq y^\alpha \\ (n,\ell) = 1}} \frac{1}{n} +
	\sum_{\substack{n \leq y^\alpha \\ (n,\ell) = 1}} \frac{\xi(n) - 1}{n} .
\end{equation}
Now, $\xi(n) = 1$ whenever $P^+(n) \leq y$, so we have
\begin{equation}
\label{eq:LowerBdOnM3}
\left|
	\sum_{\substack{n \leq y^\alpha \\ (n,\ell) = 1}} \frac{\xi(n) - 1}{n}
\right|
\leq
2 \sum_{\substack{n \leq y^\alpha \\ P^+(n) > y \\ (n,\ell) = 1}}
	\frac{1}{n} .
\end{equation}
Combining \eqref{eq:LowerBdOnM2} and \eqref{eq:LowerBdOnM3} gives
\[
\left|
	\sum_{\substack{n \leq y^\alpha \\ (n,\ell) = 1}} \frac{\xi(n)}{n}
\right|
\geq
\sum_{\substack{n \leq y^\alpha \\ (n,\ell) = 1}} \frac{1}{n} -
2 \sum_{\substack{n \leq y^\alpha \\ P^+(n) > y \\ (n,\ell) = 1}}
	\frac{1}{n} .
\]
Getting rid of the coprimality restrictions by standard arguments and applying Lemma \ref{lem:ForVino} yields
\[
\left|
	\sum_{\substack{n \leq y^\alpha \\ (n,\ell) = 1}} \frac{\xi(n)}{n}
\right|
\geq
\frac{\varphi(\ell)}{\ell} (-2\alpha \log \alpha + 3\alpha - 2) \log y
	+ O(\log \ell \log \log \ell) .
\]
Substituting this into \eqref{eq:LowerBdOnM1} gives the lower bound
\[
M(\chi) \geq
\frac{\sqrt{k}}{\pi} (-2\alpha \log \alpha + 3\alpha - 2) \log y + O(\sqrt{q}) .
\]
The right side is maximized when $\alpha = \sqrt{e}$ and $y = n_\xi$. Making these choices and rearranging the inequality concludes the proof of Theorem \ref{thm:LowerBound}.
\end{proof}

\section{Heuristics and Conjecture \ref{conj:StrongConj}}
\label{sect:Heuristic}

Let $\chi(n) := \xi(n) \left(\frac{n}{3}\right)$, where $\xi \mod{k}$ is an odd primitive character with $3 \dnd k$. Our inequality \eqref{eq:LowerBdOnM1} implies
\[
\begin{split}
M(\chi)
&\geq
\frac{\sqrt{k}}{\pi} \cdot \frac{3}{2}
\left|
\sum_{\substack{n \leq k \\ (n,3) = 1}}
	\frac{\xi(n)}{n}
\right|
+ O(\sqrt{k}) \\
& \approx
\frac{\sqrt{k}}{\pi}
\left|
\sum_{n \leq k}
	\frac{\xi(n)}{n}
\right|
+ O(\sqrt{k}) .
\end{split}
\]
Pick any $y \leq k$ and $0 = u_0 < u_1 < \cdots < u_\ell = \frac{\log k}{\log y}$. Then we can write
\[
\sum_{n \leq k} \frac{\xi(n)}{n} =
    \sum_{y^{u_0} \le n < y^{u_1}}\frac{\xi(n)}{n} +
    \sum_{y^{u_1} \le n < y^{u_2}}\frac{\xi(n)}{n} + \cdots +
    \sum_{y^{u_{\ell-1}} \le n < y^{u_\ell}}\frac{\xi(n)}{n} .
\]
Now, suppose $\xi(n) = 1$ for all $y$-smooth integers $n$. The proportion of $n \in [y^{u_{j-1}}, y^{u_j})$ which are $y$-smooth is roughly $\rho(u_j)$, where $\rho$ is the Dickman-de Bruijn function. If $n_\xi$ is large, then $\xi$ should behave randomly on those $n$ which are not $n_\xi$-smooth. We are therefore led to guess that, when $y = n_\xi$,
\[
    \sum_{y^{u_{j-1}} \le n < y^{u_j}} \frac{\xi(n)}{n}
        \approx \rho(u_j)
			\Big(\log y^{u_j} - \log y^{u_{j-1}}\Big) .
\]
This would imply
\[
\sum_{n \leq k} \frac{\xi(n)}{n} \approx
	(\log n_\xi) \sum_{j \leq \ell} (u_j - u_{j - 1}) \rho(u_j)
\approx
	(\log n_\xi) \int_0^{\frac{\log k}{\log n_\xi}} \rho(u) du .
\]
Since we presumably have $n_\xi \ll k^{o(1)}$, we see that
\[
\sum_{n \le k}\frac{\xi(n)}{n} \approx
(\log n_\xi) \int_0^\infty \rho(u) du =
e^\gamma \log n_\xi.
\]
Substituting this into our initial inequality, we are led to the statement of Conjecture \ref{conj:StrongConj}:
\[
\log n_\xi \leq
\left(\frac{\pi}{e^\gamma} + o(1)\right)
	\frac{M(\chi)}{\sqrt{k}} .
\]
Applying the conjectured bound \eqref{eq:ConjCharBound} to the even character $\chi \mod{3k}$ would yield
\[
n_\xi \ll (\log k)^{1 + o(1)} .
\]

\section{The case of the even character}
\label{sect:EvenChar}

Above, we showed that any improvement of P\'olya--Vinogradov for all even
primitive characters leads to improved bounds on $n_\xi$ for odd characters
$\xi$. What can we say about $n_\xi$ when $\xi$ is an even character?
Unfortunately, not much. However, we point out that the developing theory of
pretentiousness (pioneered by Granville and Soundararajan in \cite{GS07}) can
be used to demonstrate that for primitive even characters $\chi\mod{q}$ at least one of
$M(\chi)$ or $n_\chi$ must be small. As an illustration, we prove
\begin{theorem}
Fix $\epsilon > 0$. Then for any primitive even character $\chi\mod{q}$ with $q$ sufficiently large, we have
$n_\chi \leq \exp\Big((\log q)^{5/6 + \epsilon}\Big)$
or
$
M(\chi) \leq \sqrt{q} (\log q)^{2/3 + \epsilon} .
$
\end{theorem}
\begin{proof}
Suppose $n_\chi > \exp\Big((\log q)^{5/6 + \epsilon}\Big)$. Then
\begin{equation}
\label{eq:TBD}
\begin{split}
\D(\chi,\chi_0; q)^2
:=&
\sum_{p \leq q} \frac{1 - \Re \chi(p)}{p} \\
=&
\sum_{n_\chi \leq p \leq q}
\frac{1 - \Re \chi(p)}{p} \\
\leq&
\sum_{n_\chi \leq p \leq q} \frac{2}{p} \\
=&
\left(\frac{1}{3} - 2\epsilon\right) \log \log q +
o(1).
\end{split}
\end{equation}
In Lemma 3.3 of \cite{Balog-Granville-Sound}, Balog, Granville and
Soundararajan show that for all $\psi$ other than the most pretentious
character,
\[
    \D(\chi(n),\psi(n) n^{it}; q)^2 \geq
        \left(\frac{1}{3} + o(1)\right) \log \log q
\]
for all $t$. It follows that for $q$ sufficiently large, $\chi$ must be most pretentious to the trivial
character. If $\chi$ is even, then Theorem 2.9 of \cite{G} implies that
\[
M(\chi) \leq \sqrt{q} (\log q)^{2/3 + o(1)} .
\]
Once more assuming $q$ is large enough, we deduce the claimed bound on $M(\chi)$.
\end{proof}

%%%%%%%%%%%%%%%%%
%%%%%%%%%%%%%%%%%

%\bibliographystyle{alpha}
%\begin{thebibliography}{99}

\end{document}